\documentclass{article}
\usepackage{amssymb,amsmath}
\usepackage{graphics}

\def\qed{\hfill {\hbox{${\vcenter{\vbox{               
   \hrule height 0.4pt\hbox{\vrule width 0.4pt height 6pt
   \kern5pt\vrule width 0.4pt}\hrule height 0.4pt}}}$}}}

\def\hat{\widehat}

\newtheorem{theorem}{Theorem}
\newtheorem{definition}{Definition}
\newtheorem{lemma}[theorem]{Lemma}

\newtheorem{example}{Example}
\newtheorem{remark}[example]{Remark}

\newenvironment{proof}[1][Proof]{\smallskip\noindent{\bf #1.}\quad}%
{\qed\par\medskip}

\date{}

\title{\Large \textbf{Semiquandles and flat virtual knots}}

\author{Allison Henrich \and Sam Nelson} 

\begin{document}
\maketitle

\begin{abstract}
We introduce an algebraic structure we call \textit{semiquandles}
whose axioms are derived from flat Reidemeister moves. Finite
semiquandles have associated counting invariants and enhanced invariants
defined for flat virtual knots and links. We also introduce 
\textit{singular semiquandles} and \textit{virtual singular semiquandles}
which define invariants of flat singular virtual knots and links. 
As an application, we use semiquandle invariants to compare two
Vassiliev invariants.
\end{abstract}



\section{\large \textbf{Introduction}}

Recent works such as \cite{K} take a combinatorial approach to knot 
theory in which knots and links are regarded as equivalence
classes of knot and link diagrams. New types of combinatorial knots 
and links can then be defined by introducing new types of crossings
and Reidemeister-style moves that govern their interactions.
These new combinatorial classes of knots and links have various topological
and geometric interpretations relating to simple closed curves in
3-manifolds, rigid vertex isotopy of graphs, etc.

A \textit{flat crossing} is a classical crossing in which we ignore the
over/under information. A flat knot or link  is a planar projection or 
shadow of a knot or link on the surface on which the knot or link diagram 
is drawn. Every classical knot diagram may be regarded as a decorated or 
\textit{lift} of a flat knot, and conversely every classical knot diagram 
has a corresponding flat \textit{shadow.}

At first glance, flat knots might seem uninteresting since flattening
classical crossings apparently throws away the information which defines 
knotting. However, a little thought reveals potential applications of
flat crossings: invariants of links with classical intercomponent
crossings and flat intracomponent crossings are related to link homotopy
and Milnor invariants of ordinary classical links, for example.

Another place where flat crossings prove useful is in
virtual knot theory. Every purely flat knot is trivial, i.e.,
reducible by flat Reidemeister moves to the unknot. However, 
flat virtual knots and links (i.e., diagrams with virtual and flat
crossings) are generally non-trivial. Non-triviality of a flat virtual 
says that no choice of classical crossing information for the flat crossings
yields a classical knot. Hence, flat crossings are useful in the study
of non-classicality in virtual knots.

A \textit{singular crossing} is a crossing where two strands are fused
together. Singular knots and links may be understood as rigid vertex isotopy
classes of knotted and linked graphs, and they play a
role in the study of Vassiliev invariants of classical knots and links.

In this paper, we define an algebraic structure we call a \textit{semiquandle}
which yields counting invariants for flat virtual knots. The paper is
organized as follows. In section \ref{fvs} we define flat, singular and 
virtual knots and links. In section \ref{sq} we define semiquandles and give
some examples. In section \ref{ssq} we define singular semiquandles by
including operations at singular crossings. 
In section \ref{vssq} we define virtual semiquandles and virtual singular
semiquandles by including an operation at virtual crossings.
In section \ref{ci} we give examples to show that the 
counting invariants with respect to finite semiquandles can distinguish
flat virtual knots and links. In section \ref{A} we give an application to 
computing Vassiliev invariants of virtual knots. In section \ref{q} we 
collect some questions for further research.

\section{\large \textbf{Flat knots, virtual knots and singular knots}}
\label{fvs}

Let us introduce several types of knots we will discuss in this paper. We assume
that all knots are oriented unless otherwise specified. The simplest type of
knot of those we consider, a \emph{flat knot}, is an immersion of $S^1$ in
$\mathbf{R}^2$. Alternatively, a flat knot can be described as an equivalence
class of knot diagrams where under/over strand information at each crossing is
unspecified. The equivalence relation is given by flat versions of the
Reidemeister moves. Here, we illustrate the flat Reidemeister moves.

\[\scalebox{0.25}{\includegraphics{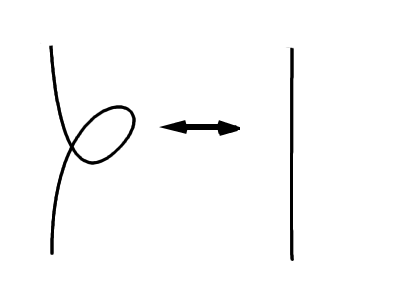}} \quad
\scalebox{0.25}{\includegraphics{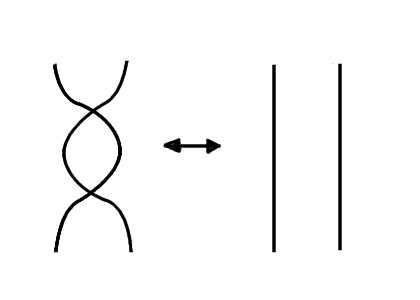}} \quad
\scalebox{0.25}{\includegraphics{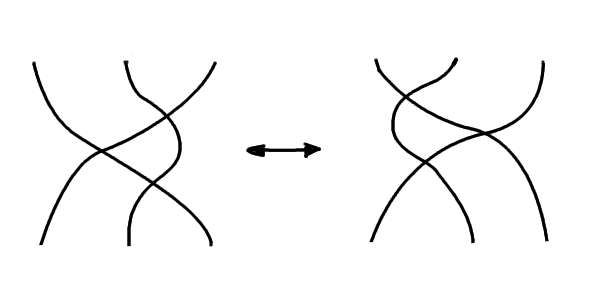}}\]

It is an easy exercise to show that any flat knot is related to the trivial flat
knot (i.e. the flat knot with no crossings) by a sequence of flat Reidemeister
moves. While the theory of flat knots appears uninteresting, if we consider the
analogous theory of flat virtual knots, we enter a highly non-trivial category.

A \emph{flat virtual knot} is a decorated immersion of $S^1$ in $\mathbf{R}^2$,
where each crossing is decorated to indicate that it is either flat or virtual.
(Virtual crossings are pictured by an encircled flat crossing.) Once again, we
may also describe a flat virtual knot as an equivalence class of virtual knot
diagrams where under/over strand information at each classical crossing is
unspecified. The corresponding equivalence relation is given by the flat
versions of the virtual Reidemeister moves in addition to the flat versions of
the ordinary Reidemeister moves.

\[\scalebox{0.25}{\includegraphics{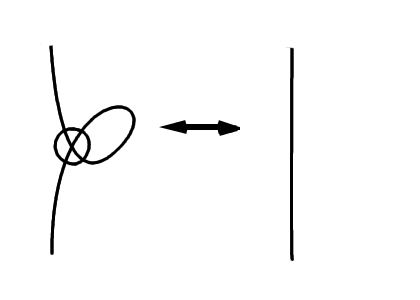}} \quad
\scalebox{0.25}{\includegraphics{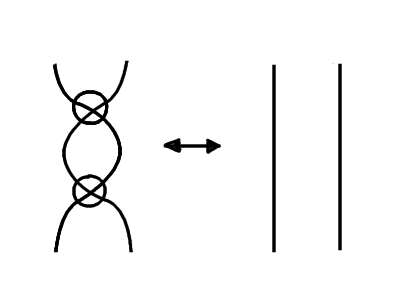}}\]
\[\scalebox{0.25}{\includegraphics{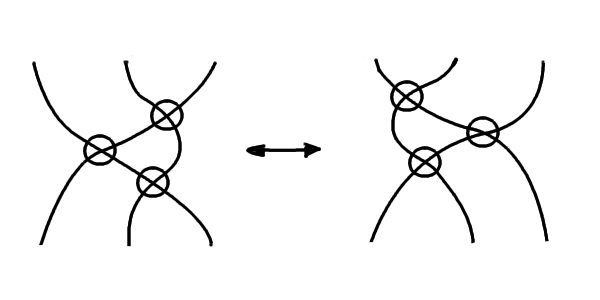}} \quad
\scalebox{0.25}{\includegraphics{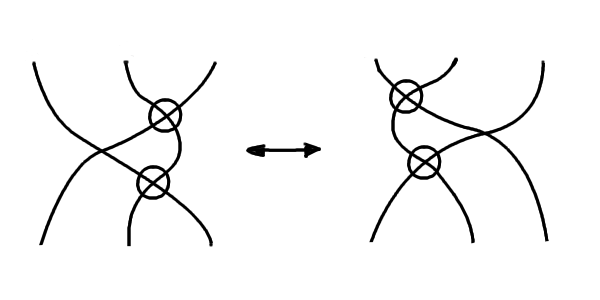}}\]

Note that the following move is forbidden.

\[\scalebox{0.25}{\includegraphics{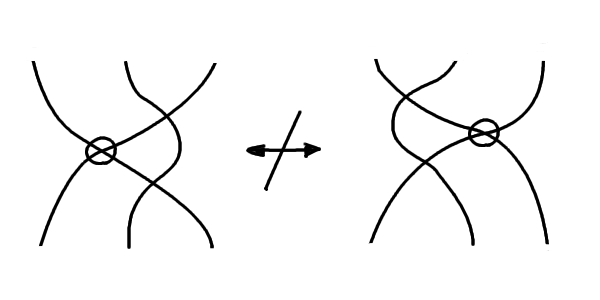}}\]

As with ordinary virtual knots, flat virtual knot diagrams have a geometric
interpretation as flat knot diagrams on surfaces. In this case, the virtual
crossings are interpreted as artifacts of a projection of the knot diagram on
the surface to a knot diagram in the plane~\cite{K}.

Finally, we'd like to consider flat knots and flat virtual knots that have
\emph{singularities}. These singularities should be thought of as rigid
vertices, or places where the knot is actually glued to itself. Thus,
\emph{flat singular knots} are simply equivalence classes of flat knots where
some crossings are decorated to indicate that they are singular. The
Reidemeister moves corresponding to flat singular equivalence are the ordinary
flat equivalence moves together with the following two moves.

\[\scalebox{0.25}{\includegraphics{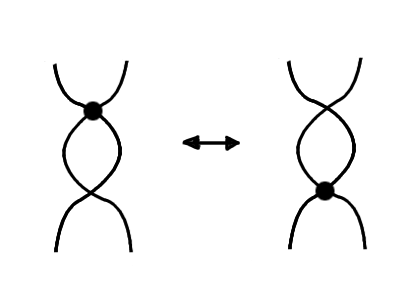}} \quad
\scalebox{0.25}{\includegraphics{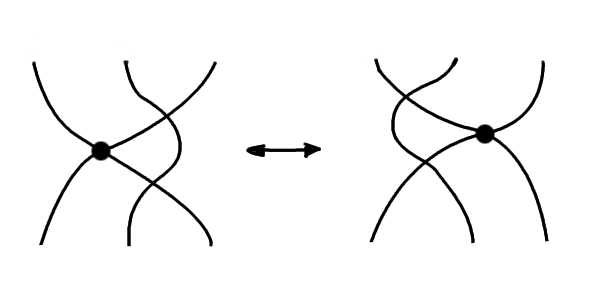}}\]

Similarly, \emph{flat virtual singular knots} are equivalence classes of flat
virtual knots where some of the  crossings may be designated as singular.
Hence, there are three types of crossings that may be contained in a diagram of
a flat virtual singular knot. The equivalence relation is given by all of the
previous flat, virtual, and singular moves together with the following move.

\[\scalebox{0.25}{\includegraphics{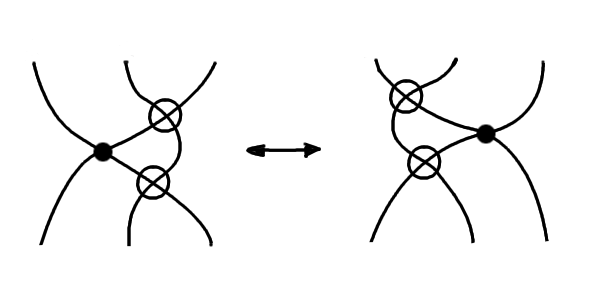}}\]

We call the simplest non-trivial flat virtual singular knot that contains all
three types of crossings the \textit{Triple Crazy Trefoil.} This is the knot 
pictured below.

\[\scalebox{0.25}{\includegraphics{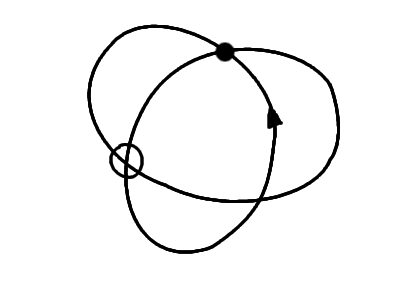}}\]

\section{\large \textbf{Semiquandles}} \label{sq}

A number of algebraic structures have been defined in recent years with
axioms derived from variations on the Reidemeister moves. The earliest of 
these is the \textit{quandle} (see \cite{J,M}) in which we have generators
corresponding to arcs in a link diagram and an invertible binary operation
at crossings. 

Subsequent papers have generalized this idea in various ways. In \cite{FR}
ambient isotopy is replaced with framed isotopy to define \textit{racks}.
In \cite{KR, FRS}, arcs in an oriented knot diagram are replaced with 
\textit{semiarcs} to define \textit{biquandles}. In \cite{KM}, an operation
at virtual crossings is included in the biquandle definition to yield 
\textit{virtual biquandles}. 

\begin{definition}
\textup{A \textit{semiquandle} is a set $X$ with two binary
operations $(x,y)\mapsto x^y, x_y$ such that for all $x,y,z\in X$ we have
\begin{list}{}{}
\item[(0)]{for all $x,y\in X$
there are \textbf{unique} $w,z\in X$ with $x=w^y$ and $x=z_y$,}
\item[(i)]{$x_y=y$ iff $y^x = x$,}
\item[(ii)]{$(x_y)^{(y^x)}=x$ and $(x^y)_{(y_x)}=x$, and}
\item[(iii)]{$(x^y)^z=(x^{z_y})^{y^z}$, $(y_x)^{z_{x^y}}=(y^z)_{x^{z_y}}$, 
and $(z_{x^y})_{y_x}=(z_y)_x$.}
\end{list}}
\end{definition}

Axiom (0) says the actions $x\mapsto x^y$ and $x\mapsto x_y$ are invertible.
The unique $z,w$ in axiom (0) will be denoted $z=x_{y^{-1}}$ and $w=x^{y^{-1}}$.
These axioms come from dividing an oriented flat knot into semiarcs, i.e. 
edges between vertices in the flat diagram regarded as a graph, and then
translating the flat Reidemeister moves into algebraic axioms. 

\[\includegraphics{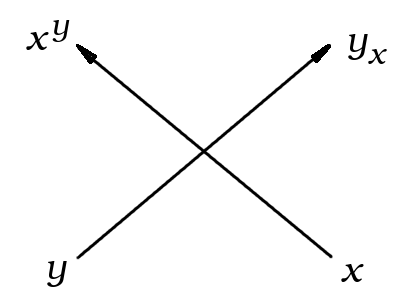}\]

In the first Reidemeister move, 
right-invertibility guarantees the uniqueness of $y$ given $x$, and the
relationship between $x$ and $y$ becomes axiom (i).
\[\includegraphics{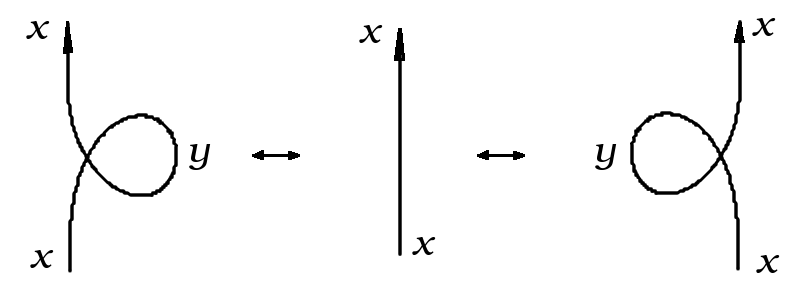} \] 

The direct II move, in which both strands are oriented in the same direction,
give us axiom (ii). Given axiom (0), the reverse II move yields the same
relationship between $x$ and $y$ where the uncrossed strands are labeled
$x$ and $y^x$.
\[\includegraphics{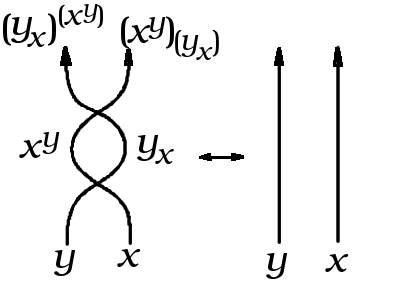} \quad \includegraphics{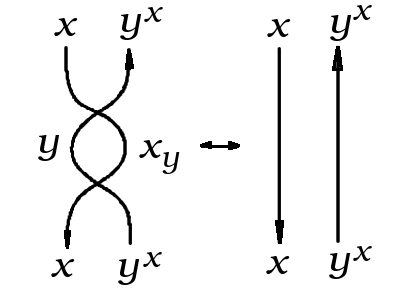} \]

Reidemeister move III yields the three equations in axiom (iii).
\[\includegraphics{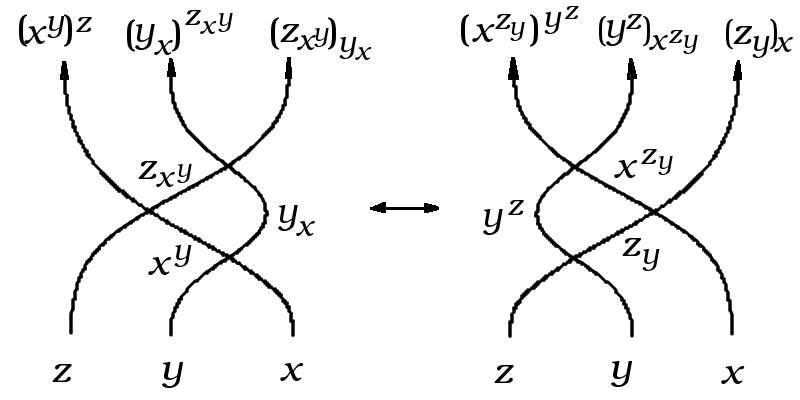}\]

\begin{definition}
\textup{For any flat virtual link $L$, the \textit{fundamental semiquandle}
$FSQ(L)$ of $L$ is the set of equivalence classes of semiquandle words
in a set of generators corresponding to semiarcs in a diagram $D$ of $L$, i.e.
edges in the graph obtained from $D$ by considering flat crossings as 
vertices, under the equivalence relation generated by the semiquandle
axioms and the relations at the crossings. As with the knot quandle, 
fundamental rack and knot biquandle, we can express the fundamental 
semiquandle with a presentation read from a diagram.}
\end{definition}

\begin{example}
\textup{The pictured \textit{flat Kishino knot} has the listed fundamental 
semiquandle presentation.}
\[\raisebox{-0.5in}{\includegraphics{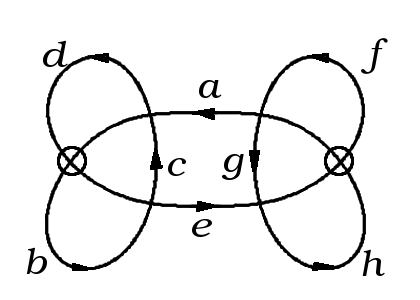}} \quad 
\begin{array}{rl} FSQ(K)=
\langle a, b, c, d, e, f, g, h \ | \ &  a^c=b,\ c_a=d,\ b^d=e, \ d_b=c,\\
 &  e^g=f,\ g_e=h, \ f^h = g, h_f=a  \rangle\end{array}\]
\end{example}

\begin{remark}
\textup{An alternative definition for the fundamental semiquandle of a flat
virtual knot is that $FSQ(L)$ is the quotient of the (strong) knot 
biquandle of any lift of $L$ (i.e., choice of classical crossing type 
for the flat crossings of $L$) under the equivalence relation generated by 
setting $a^{\overline{b}}\sim a_{b}$ and $a_{\overline{b}}\sim a^{b}$ for all
$a,b\in B(L)$. Indeed, this operation yields a ``flattening'' functor 
$SQ:\mathcal{B}\to \mathcal{S}$
from the category of strong biquandles to the category of semiquandles.}
\end{remark}

\begin{example}
\textup{For any set $X$ and bijection $\sigma:X\to X$, the operations
$x^y=\sigma(x)$ and $x_y=\sigma^{-1}(x)$ define a semiquandle
structure on $X$. We call this type of semiquandle a \textit{constant
action semiquandle} since the actions of $y$ on $x$ is constant as
$y$ varies.}
\end{example}

As is the case with quandles and biquandles (see \cite{HN,NV}), for a 
finite semiquandle
$X=\{x_1,\dots,x_n\}$ we can conveniently express the semiquandle
structure with a block matrix $M_X=[U|L]$ where $U_{i,j}=k$ and 
$L_{i,j}=l$ for $x_k=(x_i)^{(x_j)}$ and $x_l=(x_i)_{(x_j)}$. This matrix 
notation enables us to do computations with semiquandles without
the need for formulas for $x^y$ and $x_y$.

\begin{example}
\textup{The constant action semiquandle on $X=\{1,2,3\}$ with $\sigma=(132)$
has semiquandle matrix}
\[M_X=\left[\begin{array}{ccc|ccc}
3 & 3 & 3 & 2 & 2 & 2 \\
1 & 1 & 1 & 3 & 3 & 3 \\
2 & 2 & 2 & 1 & 1 & 1
\end{array}\right].\]
\end{example}

\begin{example}\label{sq1}
\textup{Any (strong) biquandle in which $a^{\overline{b}}=a_b$ and 
$a_{\overline{b}}=a^b$ is a semiquandle. Indeed, an alternative name 
for semiquandles might be \textit{symmetric biquandles}. An example
of a non-constant action semiquandle found in \cite{NV} is}
\[M_T=\left[\begin{array}{cccc|cccc}
1 & 4 & 2 & 3 & 1 & 3 & 4 & 2 \\
2 & 3 & 1 & 4 & 3 & 1 & 2 & 4 \\
4 & 1 & 3 & 2 & 2 & 4 & 3 & 1 \\
3 & 2 & 4 & 1 & 4 & 2 & 1 & 3
\end{array}\right].\]
\end{example}

\section{\large \textbf{Singular semiquandles}}
\label{ssq}

Let us now consider what happens to our algebraic structure when we
allow singular crossings in an oriented flat virtual knot. As with
flat crossings, we define two binary operations at a singular crossing. 
One notable difference is that unlike flat crossings, singular crossings 
are permanent -- there are no moves which either introduce or remove 
singular crossings. Indeed, the number of singular crossings is an 
invariant of singular knot type. In particular, we do not need 
right-invertibility for our operations at singular crossings.

\[\includegraphics{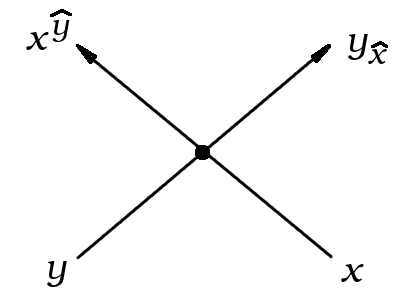}\]

\begin{definition}
\textup{Let $X$ be a semiquandle. A \textit{singular semiquandle} structure
on $X$ is a pair of binary operations on $X$ denoted $(x,y)\mapsto x^{\hat{y}},
x_{\hat{y}}$ satisfying for all $x,y,z\in X$
\begin{list}{}{}
\item[(hi)]{$(y_x)^{(\hat{x^y})}=(y_{\hat{x}})^{(x^{\hat{y}})}$ and 
$(x^y)_{(\hat{y_x})}=(x^{\hat{y}})_{(y_{\hat{x})}}$}
\item[(hii)]{$(x^y)^{\hat{z}}=(x^{\hat{z_y}})^{y^z}$, $(y_x)^{{z_{\hat{x^y}}}}= 
(y^z)_{x^{\hat{z_y}}}$ and $(z_{\hat{x^y}})_{y_x}=(z_y)_{\hat{x}}. $}
\end{list}}
\end{definition}

We call axioms (hi) and (hii) the \textit{hat axioms} for the obvious 
reason. These axioms come from the subset of the oriented
singular flat Reidemeister moves pictured below.
\[\includegraphics{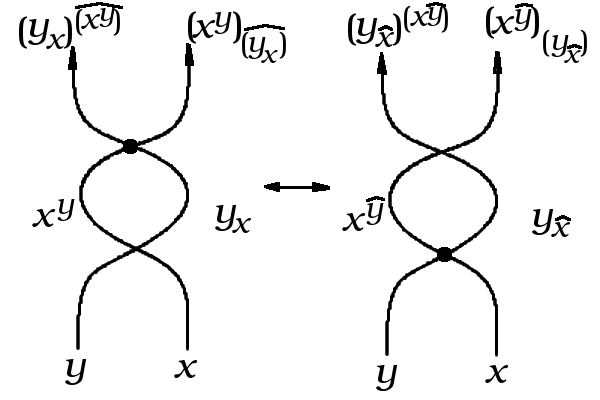} \quad \includegraphics{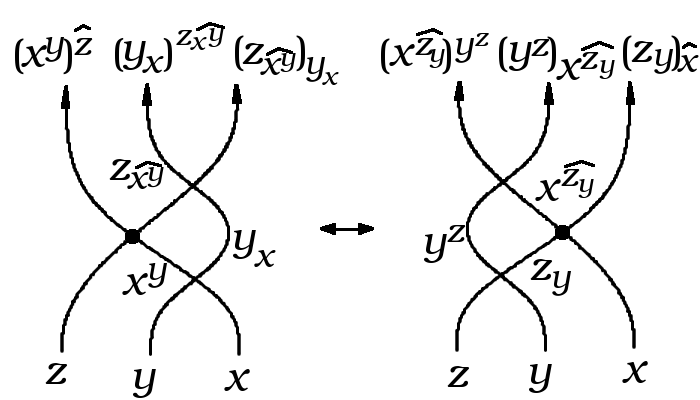}\]

To see that the two pictured oriented singular  moves are sufficient to give 
us all of the oriented flat singular moves, we note the following key lemmas.

\bigskip

\begin{lemma}
The move \raisebox{-0.5in}{\includegraphics{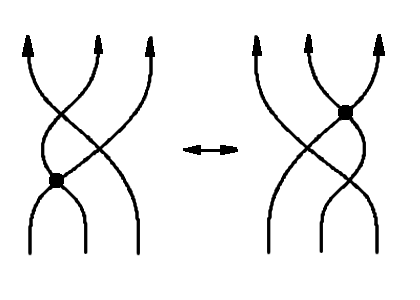}}
follows from the flat Reidemeister moves and the two pictured moves.
\end{lemma}

\begin{proof}

\[\raisebox{-0.5in}{\scalebox{0.8}{\includegraphics{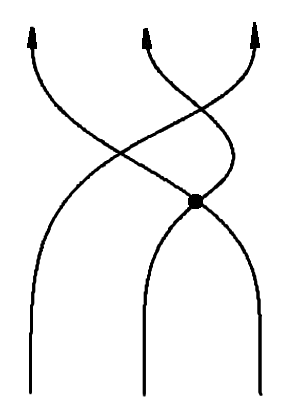}}} 
\longleftrightarrow
\raisebox{-0.5in}{\scalebox{0.8}{\includegraphics{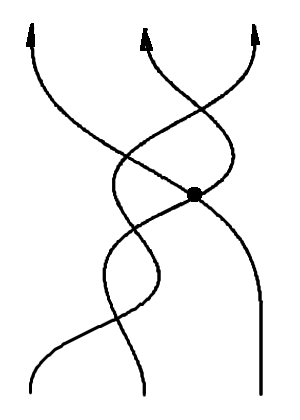}}} 
\longleftrightarrow
\raisebox{-0.5in}{\scalebox{0.8}{\includegraphics{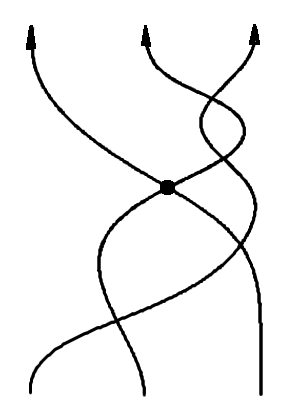}}} 
\longleftrightarrow
\raisebox{-0.5in}{\scalebox{0.8}{\includegraphics{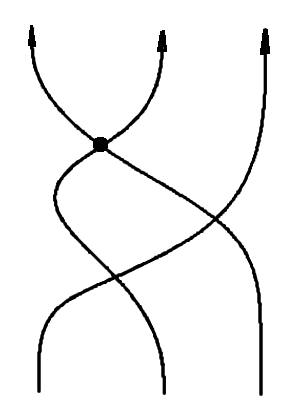}}}
\]

\end{proof}

Similar move sequences yield the other oriented flat/singular type III moves.

\begin{lemma}
The reverse oriented singular II move follows from the flat moves and
the moves pictured above.
\end{lemma}

\begin{proof}

Starting with one side of a reverse singular II move, we can use flat
moves to get a symmetrical diagram in which we can apply a direct singular
II move.
\[\raisebox{-0.5in}{\scalebox{0.8}{\includegraphics{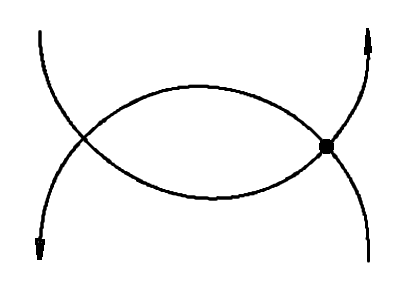}}} 
\longleftrightarrow
\raisebox{-0.5in}{\scalebox{0.8}{\includegraphics{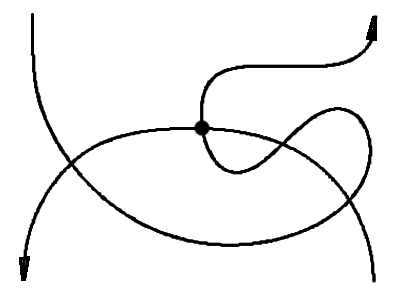}}} 
\longleftrightarrow
\raisebox{-0.5in}{\scalebox{0.8}{\includegraphics{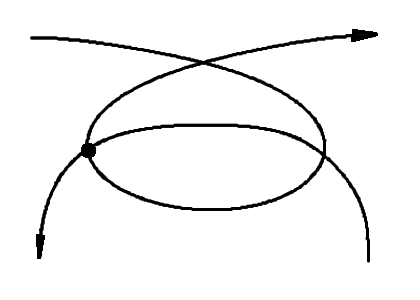}}}
\]
Reversing the process gives the other side of the reverse singular II move.
\end{proof}

\begin{example}
\textup{Let $X$ be a semiquandle. Then clearly setting
$x^{\hat{y}}=x^y$ and $x_{\hat{y}}=x_y$ for all $x,y\in X$ defines
a compatible singular structure, which we call the \textit{flat} singular
structure, $(X,X)$.}
\end{example}

\begin{example}
\textup{Let $X$ be a semiquandle. Then $a^{\hat{b}}=a_{\hat{b}}=b$ is a 
compatible singular structure, since we have}
\[(y_x)^{(\hat{x^y})}=x^y=(y_{\hat{x}})^{(x^{\hat{y}})},\quad
(x^y)_{(\hat{y_x})}=y_x=(x^{\hat{y}})_{(y_{\hat{x})}}\]
\[(x^y)^{\hat{z}}=z = (z_y)^{y^z}=(x^{\hat{z_y}})^{y^z},\quad  
(y_x)^{z_{\hat{x^y}}}= (y_x)^{x^y}= y=(y^z)_{\hat{z_y}}=(y^z)_{x^{\hat{z_y}}}\] 
\textup{and} 
\[(z_{\hat{x^y}})_{y_x}=(x^y)_{y_x}=x=(z_y)_{\hat{x}}.\]
\textup{Let us call this singular structure the \textit{operator} singular
structure on $X$, denoted $(X,O)$.}
\end{example}

As with the flat virtual case, for any flat singular virtual link $L$ there
is an associated \textit{fundamental singular semiquandle} $FSSQ(L)$ with 
presentation readable from the diagram. Elements of $FSSQ(L)$ are equivalence
classes of singular semiquandle words in generators corresponding to semiarcs
in the diagram (here we divide the diagram at both flat and singular 
crossing points, but not at virtual crossings) under the equivalence 
relation generated by the axioms (0), (i), (ii), (iii), (hi) and (hii).

\begin{example}
\textup{The \textit{triple crazy trefoil} pictured below has the listed
fundamental singular semiquandle presentation.}
\[\raisebox{-0.5in}{\includegraphics{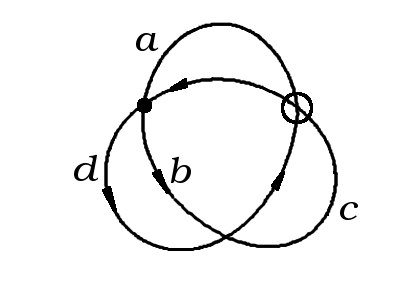}}
\langle a,\ b,\ c,\ d \ | \ a^{\hat{c}}=b,\ c_{\hat{a}}=d,\ d^b=a,\ b_d=c 
\rangle. \]
\end{example}

As with the semiquandle structure, we can represent the singular
operations in a finite singular semiquandle with matrices encoding
the operation tables. Indeed, it seems convenient to combine these 
matrices with the semiquandle operation matrices into a single block
matrix of the form 
\[M_T=\left[\begin{array}{c|c} 
i^j & i_j \\ \hline
i^{\hat j} & i_{\hat{j}}
\end{array}\right].\]

\begin{example}
\textup{The constant action semiquandle $X=\{1,2,3\}$ with $\sigma=(132)$
and operator singular structure $a^{\hat b}=a_{\hat b}=b$ has block matrix}
\[M_{(X,O)}=\left[\begin{array}{ccc|ccc}
3 & 3 & 3 & 2 & 2 & 2 \\
1 & 1 & 1 & 3 & 3 & 3 \\
2 & 2 & 2 & 1 & 1 & 1 \\ \hline
1 & 2 & 3 & 1 & 2 & 3 \\
1 & 2 & 3 & 1 & 2 & 3 \\
1 & 2 & 3 & 1 & 2 & 3 \\
\end{array}\right].\]
\end{example}

\section{\large \textbf{Virtual semiquandles and virtual singular 
semiquandles}}
\label{vssq}

As with singular crossings, we can further generalize semiquandles
by adding an operation at virtual crossings. The simplest way to do this is
to use a unary operation at each virtual crossing defined by applying 
a bijection $v$ when going through a virtual crossing from right to left
when looking in the direction of the strand being crossed and applying
$v^{-1}$ when going through a virtual crossing from left to right
when looking in the direction of the strand being crossed. 

\[\includegraphics{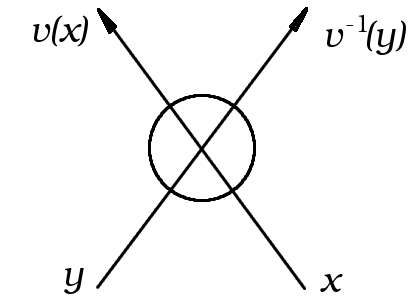} \]

As noted in
\cite{KM}, this setup ensures that the virtual I, II and III moves are 
respected by the virtual operation.

\[
\includegraphics{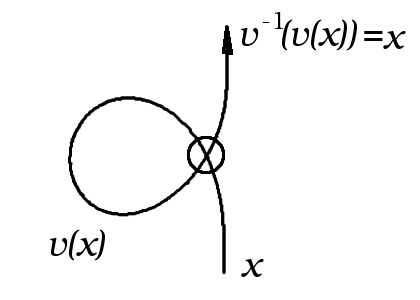} \quad
\includegraphics{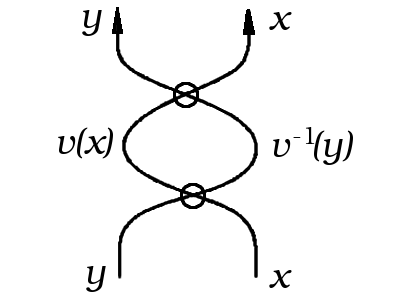} \quad
\includegraphics{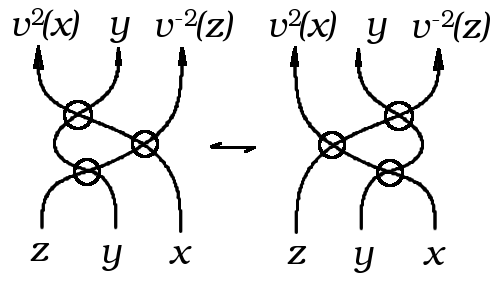} 
\]

The interaction of the virtual crossings with the flat and singular 
crossings given by the Reidemeister moves tell us how the virtual 
operation should interact with the semiquandle and singular semiquandle
structures -- namely, $v$ must be an automorphism of both structures.

\[\includegraphics{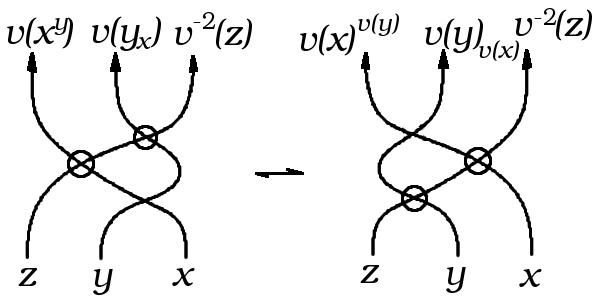} \quad
\includegraphics{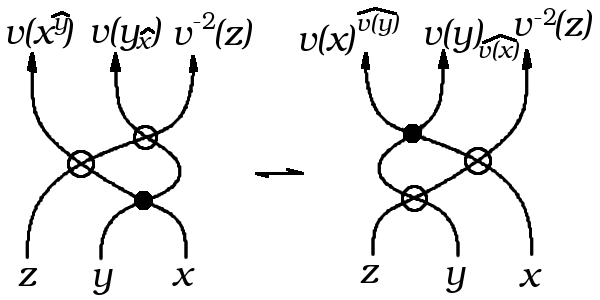}\]

\begin{definition}
\textup{A \textit{virtual semiquandle} is a semiquandle $S$ with a
choice of automorphism $v:S\to S$. A \textit{virtual singular semiquandle}
is a singular semiquandle with a semiquandle automorphism $v:S\to S$
which is also an automorphism of the singular structure. That is,
$v:S\to S$ is a bijection satisfying}
\[v(x^y)=v(x)^{v(y)}, \quad
v(x_y)=v(x)_{v(y)}, \quad
v(x^{\hat{y}})=v(x)^{\hat{v(y)}}, \quad \mathrm{and} \quad
v(x_{\hat{y}})=v(x)_{\hat{v(y)}}.
\]
\end{definition}

\begin{example}
\textup{Every semiquandle is a virtual semiquandle with $v=\mathrm{Id}_S$.
More generally, the set of virtual semiquandle structures on a semiquandle
$S$ corresponds to the set of conjugacy classes in the automorphism
group $\mathrm{Aut}(S)$ of the semiquandle $S$: let 
$v,v',\phi\in \mathrm{Aut}(S)$ with $v'=\phi^{-1}v\phi$. Then 
$\phi(S,v)\to (S,v')$ is an isomorphism of virtual semiquandles.}
\end{example}

Every flat singular virtual knot or link has a \textit{fundamental virtual 
singular semiquandle} obtained by dividing the knot or link into semiarcs
at flat, singular and virtual crossings; then $VFSSQ(L)$ has generators 
corresponding to semiarcs and relations at the crossings as determined 
by crossing type in addition to relations coming from the virtual singular
semiquandle axioms.

\section{\large \textbf{Counting invariants of flat singular virtuals}} 
\label{ci}

As with finite groups, quandles and biquandles, finite semiquandles can be
used to define computable invariants of flat virtual knots and links
by counting homomorphisms. 

\begin{definition}
\textup{Let $L$ be a flat virtual link and $T$ a finite semiquandle. The 
\textit{semiquandle counting invariant} of $L$ with respect to $T$ is
the cardinality}
\[sc(L,T)=|\mathrm{Hom}(FSQ(L),T)|\]
\textup{of the set of semiquandle homomorphisms $f:FSQ(L)\to T$ from
the fundamental semiquandle of $L$ to $T$ (i.e., maps such that 
$f(x_y)=f(x)_{f(y)}$ and $f(x^y)=f(x)_{f(y)}$ for all $x,y\in FSQ(L)$). }
\end{definition}

\begin{remark}
\textup{A semiquandle homomorphism $f:FSQ(L)\to T$ can be pictured as 
a ``coloring'' of a diagram $D$ of $L$ by $T$, i.e., an assignment of
an element of $T$ to each semiarc in $D$ such that the colors satisfy the
semiquandle operation conditions at every crossing.}
\end{remark}

\begin{example}
\textup{The semiquandle counting invariant with respect to the semiquandle
$T$ in example \ref{sq1} distinguishes the flat Kishino 
knot $FK$ from from the flat unknot $FU$ with $sc(FK,T)=16$ and
$sc(FU,T)=4$. This same semiquandle also distinguishes the flat virtual
knot $K$ from \cite{K} below from both the unknot and the flat Kishino,
with $sc(K,T)=2$.}
\[\includegraphics{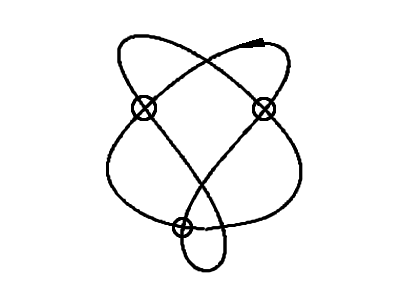}\]
\end{example}

We can enhance the semiquandle counting invariant by taking note of the 
cardinality of the image subsemiquandles $\mathrm{Im}(f)$ for each 
homomorphism to obtain a multiset-valued invariant, which we can also express
in a polynomial form by converting multiset elements to exponents of
a dummy variable $z$ and multiplicities to coefficients. Note that
specializing $z=1$ in the enhanced invariant yields the original counting
invariant.

\begin{definition}
\textup{Let $L$ be a flat virtual link and $T$ a finite semiquandle.
The \textit{enhanced semiquandle counting multiset} is the multiset }
\[sqcm(L,T)=\{\mathrm{Im}(f)\ | \ f\in \mathrm{Hom}(FSQ(L),T)) \}\]
\textup{and the \textit{enhanced semiquandle polynomial} is}
\[sqp(L,T)= \sum_{f\in \mathrm{Hom}(FSQ(L),T)} z^{|\mathrm{Im}(f)|}.\]
\end{definition}

For singular semiquandles, we also have counting invariants and polynomial
enhanced invariants.

\begin{definition}
\textup{Let $L$ be a flat singular virtual link and $(T,S)$ a finite singular 
semiquandle. Then we have the \textit{singular semiquandle counting invariant}}
\[ssc(L,(T,S))=|\mathrm{Hom}(FSSQ(L),(T,S))|,\]
\textup{the \textit{enhanced singular semiquandle counting multiset}}
\[ssqcm(L,(T,S))=\{\mathrm{Im}(f)\ | \ f\in \mathrm{Hom}(FSSQ(L),(T,S))) \},\]
\textup{and the \textit{enhanced singular semiquandle polynomial}}
\[ssqp(L,(T,S))= \sum_{f\in \mathrm{Hom}(FSSQ(L),(T,S))}z^{|\mathrm{Im}(f)|}.\]
\end{definition}

\begin{example}
\textup{The constant action semiquandle $(X_{(132)},O)$ with operator
singular structure distinguishes the triple crazy trefoil $TCT$ from the 
singular knot with one singular crossing and no other crossings $SU_1$:}
\[
\begin{array}{cc}
\scalebox{0.2}{\includegraphics{TripleCrazyTrefoil.png}} &
\includegraphics{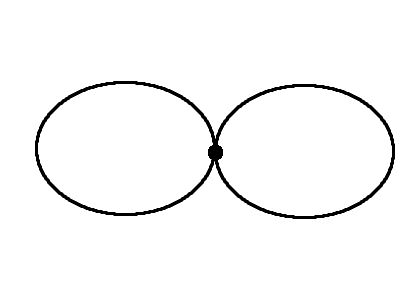}  \\
ssqp(TCT,(X_{(132)},O))=0 & ssqp(SU_1,(X_{(132)},O))= 9z^3
\end{array}
\]
\end{example}

Finally, we have counting invariants for flat singular virtual knots and
links defined analogously using finite virtual singular semiquandles.

\begin{definition}
\textup{Let $L$ be a flat singular virtual link and $(T,S,v)$ a finite virtual 
singular semiquandle. Then we have the \textit{virtual singular semiquandle 
counting invariant}}
\[vssc(L,(T,S,v))=|\mathrm{Hom}(FVSSQ(L),(T,S,v))|,\]
\textup{the \textit{enhanced virtual singular semiquandle counting multiset}}
\[vssqcm(L,(T,S,v))=\{\mathrm{Im}(f)\ | \ f\in \mathrm{Hom}(FVSSQ(L),(T,S,v)) 
\},\]
\textup{and the \textit{enhanced virtual singular semiquandle polynomial}}
\[vssqp(L,(T,S,v))= \sum_{f\in \mathrm{Hom}(FVSSQ(L),(T,S,v))}z^{|\mathrm{Im}(f)|}.\]
\end{definition}

\begin{example}
\textup{The \textit{flat virtual Hopf link} $fH$ below is distinguished from 
the flat unlink of two components by the counting invariants with respect to the
listed virtual semiquandle. Note that we can regard $T$ as a flat singular
virtual semiquandle with trivial singular operations $x^{\hat{y}}=x=x_{\hat{y}}$.}
\[M_{T,S}=
\left[\begin{array}{ccc|ccc}
1 & 3 & 1 & 1 & 3 & 1 \\
2 & 2 & 2 & 2 & 2 & 2 \\
3 & 1 & 3 & 3 & 1 & 3 \\ 
\end{array}\right], \quad v=(13)\]
\[\begin{array}{cc}
\includegraphics{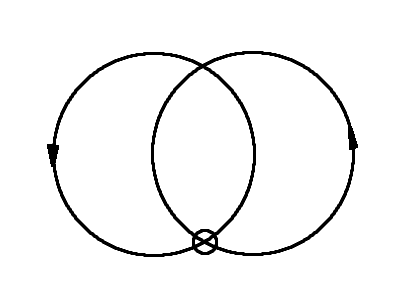} & \includegraphics{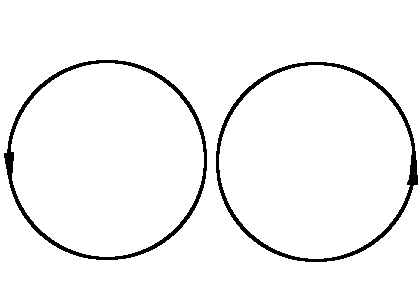} \\
vsqp(fvH,(T,S,v)) = q+4z^2 & vsqp(U_2,(T,S,v)) = q+4z^2+4z^3 \\
\end{array}
\]
\end{example}

\begin{remark}\textup{
We note that a virtual semiquandle is a virtual singular semiquandle
with trivial singular structure, i.e $x^{\hat{y}}=x_{\hat{y}}=x$, a singular
semiquandle is a virtual singular semiquandle with trivial virtual
operation, i.e. $v=\mathrm{Id}$, and a semiquandle is a virtual singular 
semiquandle with trivial virtual and singular structures.}
\end{remark}

\section{\large \textbf{Application to Vassiliev invariants}}\label{A}

In~\cite{H}, we find several degree one Vassiliev invariants for virtual 
knots. One invariant, $\mathbf{S}$, takes its values in the free abelian 
group on the set of two-component flat virtual links. Another invariant, 
$\mathbf{G}$, takes its values in the free abelian group on the set of 
flat virtual singular knots with one singularity. It is easy to show that 
$\mathbf{G}$ is at least as strong as $\mathbf{S}$, but somewhat difficult
 to show that $\mathbf{G}$ is strictly stronger than $\mathbf{S}$. Here, 
we give the definitions of these invariants and provide an alternative 
proof that $\mathbf{G}$ is strictly stronger than $\mathbf{S}$.

\begin{definition} Let $K$ be a virtual knot with diagram $\widetilde K$. 
Let $\widetilde K^d_{smooth}$ be the flat virtual link obtained by smoothing 
$\widetilde K$ at the crossing $d$ and projecting onto the associated flat 
virtual link. Furthermore, let $\widetilde K^0_{link}$ be the flat virtual 
link obtained by taking the flat projection of $\widetilde K$ disjoint union 
with the unknot. If $[L]$ represents the generator of the free abelian group 
on the set of two-component flat virtual links associated to the link $L$, 
then 
\[\mathbf{S}(K)=\sum _d sign(d)([\widetilde K ^d_{smooth}]-
[\widetilde K^0_{link}]).\] 
Here, the sum ranges over all classical crossings 
in $\widetilde K$, and $sign(d)$ is the local writhe.
\end{definition}

Since this ``smoothing" invariant has values involving flat virtual links, 
it is clear that semiquandles may be of use in computing $\mathbf{S}$ for 
pairs of virtual knots. Moreover, singular semiquandles can be put to use 
when computing the following invariant.

\begin{definition} Let $K$ be a virtual knot with diagram $\widetilde K$. 
Let $\widetilde K^d_{glue}$ be the flat virtual singular knot obtained by 
gluing $\widetilde K$ at the crossing $d$ and projecting onto the associated 
flat virtual singular knot. Let $\widetilde K^0_{sing}$ be the flat virtual 
singular knot obtained by taking the flat projection of $\widetilde K$, 
introducing a kink via the flat Reidemeister 1 move, and gluing at the 
resulting crossing. If $[L]$ represents the generator of the free abelian 
group on the set of flat virtual singular knots associated to the knot $L$, 
then 
\[\mathbf{G}(K)=\sum _d sign(d)([\widetilde K ^d_{glue}]-
[\widetilde K^0_{sing}]).\]
Here again, the sum ranges over all classical 
crossings in $\widetilde K$, and $sign(d)$ is the local writhe.
\end{definition}

It is proven in~\cite{H} that both $\mathbf{S}$ and $\mathbf{G}$ are degree 
one Vassiliev invariants and $\mathbf{G}$ is at least as strong as 
$\mathbf{S}$. To show that $\mathbf{G}$ is stronger than $\mathbf{S}$, 
consider the following pair of virtual knots.

\[
\begin{array}{cc}
\scalebox{0.3}{\includegraphics{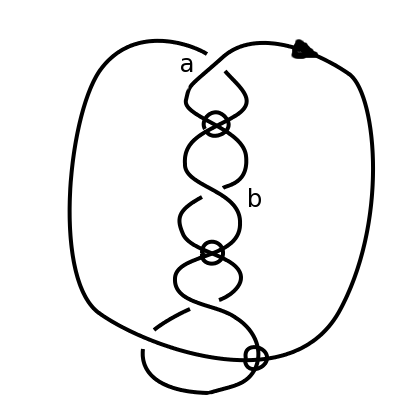}} &
\scalebox{0.3}{\includegraphics{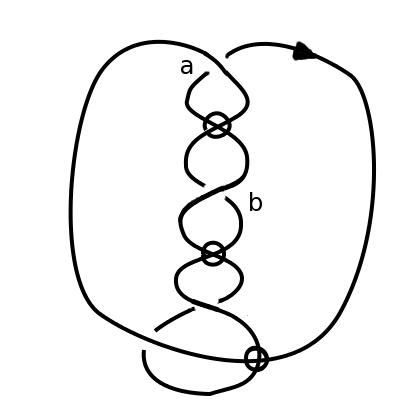}}  \\
\end{array}
\]

Let us call the first of these knots $K_1$ and the second $K_2$. Since the 
only difference between the two knots is the signs of the crossings labelled 
$a$ and $b$, we see that 
\[\mathbf{S}(K_2)-\mathbf{S}(K_1)= 2([\widetilde K ^a_{smooth}]-
[\widetilde K ^b_{smooth}])\]

and  

\[\mathbf{G}(K_2)-\mathbf{G}(K_1)= 2([\widetilde K ^a_{glue}]-
[\widetilde K ^b_{glue}]).\]

Now $\widetilde K ^a_{smooth}$ is the same as $\widetilde K ^b_{smooth}$. They 
are both the flat virtual link pictured below.

\[\includegraphics{ah-sn-23.png}\]

It follows that $\mathbf{S}(K_1)=\mathbf{S}(K_2)$. On the other hand, we can 
show using singular semiquandles that $\widetilde K ^a_{glue}$ and 
$\widetilde K ^b_{glue}$, as pictured below, are distinct.

\[
\begin{array}{cc}
\scalebox{0.3}{\includegraphics{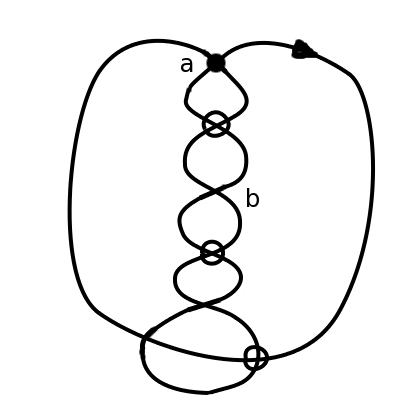}} &
\scalebox{0.3}{\includegraphics{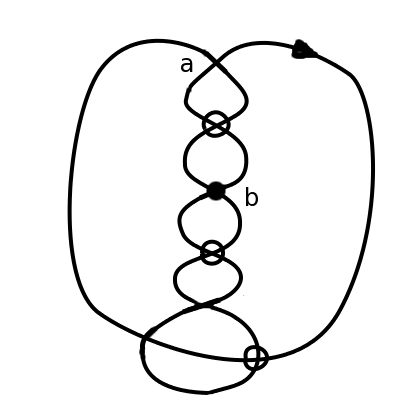}}  \\
\end{array}
\]

Consider the following singular semiquandle, $T$, given in terms of its 
matrix $M$.

\[M_=\left[\begin{array}{cccc|cccc}
1 & 4 & 2 & 3 & 1 & 3 & 4 & 2 \\
2 & 3 & 1 & 4 & 3 & 1 & 2 & 4  \\
4 & 1 & 3 & 2 & 2 & 4 & 3 & 1 \\
3 & 2 & 4 & 1 & 4 & 2 & 1 & 3 \\ \hline
1 & 1 & 4 & 4 & 1 & 2 & 2 & 1 \\
1 & 1 & 4 & 4 & 4 & 3 & 3 & 4  \\
2 & 2 & 3 & 3 & 4 & 3 & 3 & 4 \\
2 & 2 & 3 & 3 & 1 & 2 & 2 & 1 \\
\end{array}\right]\]

The enhanced singular semiquandle polynomial for $\widetilde K ^a_{glue}$ is 
$ssqp(\widetilde K ^a_{glue},T)=2z$ while the polynomial for 
$\widetilde K ^b_{glue}$ is $ssqp(\widetilde K ^b_{glue},T)=2z+2z^4$. Hence, 
the two flat virtual singular knots are distinct and, thus, 
$\mathbf{G}(K_1)\neq\mathbf{G}(K_2)$.

\section{\large \textbf{Questions}}\label{q}

In this section, we collect questions for future research.

Singular semiquandles bear a certain resemblance to virtual biquandles,
in which a biquandle is augmented with operations at virtual crossings.
Given a biquandle $B$, the set of virtual biquandle structures on $B$
forms a group isomorphic to the automorphism
group of $B$. What is the structure of the set of singular semiquandle
structures on a semiquandle $X$?

Our algebra-agnostic approach to computation of our various 
semiquandle-based invariants works well for small-cardinality semiquandles
and link diagrams with small crossing numbers. However, for links with higher
crossing numbers and larger
coloring semiquandles we will need more algebraic descriptions. We have 
given a few examples of classes of semiquandle structures, e.g. constant
action semiquandles and operator singular structures. What are some examples
of group-based or module-based semiquandle and singular semiquandle 
structures akin to Alexander biquandles? (Note that the only Alexander 
biquandles which are semiquandles are constant action Alexander biquandles).

Enhancement techniques for biquandle counting invariants which should
extend to semiquandles include \textit{semiquandle cohomology} which is
the special case of Yang-Baxter cohomology described in \cite{CYB} and 
the flattened case of $S$-cohomology as described in \cite{CN}. Similarly,
we might define \textit{semiquandle polynomials} and the resulting 
enhancements of the counting invariants as in \cite{N}.What
other enhancements of semiquandle counting invariants are there?

What is the relationship, if any, between semiquandle invariants and 
quaternionic biquandle invariants described in \cite{BA}?

\bigskip

Our \texttt{python} code for computing semiquandle-based invariants is 
available from the second listed author's website at 
\texttt{www.esotericka.org.}

\bigskip

\textsc{Department of Mathematics, Oberlin College, 10 North Professor St.,
Oberlin, Ohio 44074}

\noindent \textit{Email:} \texttt{ahenrich@oberlin.edu}

\bigskip

\textsc{Department of Mathematics, Claremont McKenna College,
 850 Colubmia Ave., Claremont, CA 91711}

\noindent
\textit{Email address:} \texttt{knots@esotericka.org}


\begin{thebibliography}{0}

\bibitem{BA}{A. Bartholomew and R. Fenn.
Quaternionic invariants of virtual knots and links. 
\textit{J. Knot Theory Ramifications} \textbf{17} (2008) 231-251. }

\bibitem{CYB}{J. S. Carter, M. Elhamdadi and M. Saito. Homology Theory for 
the Set-Theoretic Yang-Baxter Equation and Knot Invariants from 
Generalizations of Quandles.  \textit{Fund. Math.} \textbf{184} (2004) 31-54.}

\bibitem{CN}{J. Ceniceros and S. Nelson. Virtual Yang-Baxter 2-cocycle invariants. arXiv:0708.4254, To appear in \textit{Trans.\ Amer.\ Math.\ Soc.}}

\bibitem{FJK}{R. Fenn, M. Jordan-Santana and L. Kauffman. Biquandles 
and virtual links.  \textit{Topology Appl.}  \textbf{145}  (2004) 157-175.}

\bibitem{FR}{R. Fenn and C. Rourke.
 Racks and links in codimension two.
 \textit{J. Knot Theory Ramifications}  \textbf{1}  (1992) 343-406.}

\bibitem{FRS}{R. Fenn, C. Rourke and B. Sanderson. 
Trunks and classifying spaces. 
\textit{Appl. Categ. Structures}  \textbf{3}  (1995) 321-356.}

\bibitem{H}{A. Henrich. A sequence of degree one Vassiliev invariants for 
virtual knots.
\textit{arXiv:0803.0754v3} [math.GT] (2008) 1-26.}

\bibitem{HN}{B. Ho and S. Nelson. Matrices and finite quandles.
\textit{Homology Homotopy Appl.} \textbf{7} (2005) 197-208.}

\bibitem{J}{D. Joyce.
 A classifying invariant of knots, the knot quandle.
 \textit{J. Pure Appl. Algebra}  \textbf{23}  (1982)  37-65.}

\bibitem{K}{L. Kauffman. Virtual Knot Theory. \textit{European J. Combin.}
\textbf{20} (1999) 663-690.}

\bibitem{KM}{L. H. Kauffman and V. O. Manturov. Virtual biquandles. 
\textit{Fundam. Math.} \textbf{188} (2005) 103-146.}

\bibitem{KR}{L. H. Kauffman and D. Radford. Bi-oriented quantum algebras, 
and a generalized Alexander polynomial for virtual links. 
\textit{Contemp. Math}. \textbf{318} (2003) 113-140.}

\bibitem{M}{S. V. Matveev.
Distributive groupoids in knot theory. 
\textit{Math. USSR, Sb.} \textbf{47} (1984) 73-83.}

\bibitem{N}{S. Nelson. Generalized quandle polynomials. 
arXiv:0801.2979, To appear in \textit{Can.\ Bull.\ Math.}}

\bibitem{NV}{S. Nelson and J. Vo. Matrices and Finite Biquandles. 
\textit{Homology, Homotopy Appl.} \textbf{8} (2006) 51-73.}


\end{thebibliography}
\end{document}